\documentclass[final,1p,times]{elsarticle}
\usepackage{amsmath}
\allowdisplaybreaks[4]
\usepackage{amsthm}
\usepackage{amscd}
\RequirePackage{CJKnumb}
\usepackage{amsfonts}
\usepackage{amssymb}
\usepackage{graphicx}
\usepackage{caption} 
\usepackage{subfigure}
\usepackage{subcaption}
\usepackage{booktabs}
\biboptions{sort&compress}
\usepackage{xcolor}
\usepackage[pagebackref, colorlinks, citecolor=green, linkcolor=red, urlcolor=blue]{hyperref}
\numberwithin{equation}{section}
\newtheorem{theorem}{Theorem}[section]
\newtheorem{lemma}[theorem]{Lemma}
\usepackage{mathrsfs}
\usepackage{titletoc}
\usepackage{float}
\usepackage{longtable}
\usepackage{adjustbox}
\journal{***}
\begin{document}
	\begin{frontmatter}
		\title{On  topological solutions to a generalized Chern-Simons equation on lattice graphs}
		\author{Songbo Hou \corref{cor1}}
		\ead{housb@cau.edu.cn}
		\address{Department of Applied Mathematics, College of Science, China Agricultural University,  Beijing, 100083, P.R. China}
		\author{Xiaoqing Kong}
		\ead{kxq@cau.edu.cn}
		\address{Department of Applied Mathematics, College of Science, China Agricultural University,  Beijing, 100083, P.R. China}
		
		\cortext[cor1]{Corresponding author: Songbo Hou}
		\begin{abstract}

	For \(n \geq 2\), consider \(\mathbb{Z}^n\) as a lattice graph. We explore a generalized  Chern-Simons equation on \(\mathbb{Z}^n\). Employing the method of exhaustion, we prove that  there exists a global solution that also qualifies as a topological solution. Our results extend those of Hua et al. [arXiv:2310.13905] and complement the findings of Chao and Hou [J. Math. Anal. Appl. $\bf{519}$(1), 126787(2023)], as well as those of Hou and Qiao [J. Math. Phys. $\bf{65}$(8), 081503(2024)].

		\end{abstract}	
	\end{frontmatter}
	
	\section{Introduction}
	
	The Chern-Simons theory was originally formulated by Shiing-Shen Chern and  James Simons in 1974. This theory was initially developed within the field of mathematics to study the geometric structures on three-dimensional manifolds. Later, it found broad applications in physics, particularly in quantum physics and condensed matter physics. The theory quickly gained significance in understanding topological phase transitions, especially in the context of quantum Hall effects, topological insulators, and high-temperature superconductors \cite{MR0855219,MR1068089,MR2793600}.
	
	 Substantial studies have been undertaken
	  on the  dynamic models of the Chern-Simons type in  the field theory,  referenced in works like \cite{lin2011non,nam2013vortex,lin2009vortex} among others.
	From a mathematical perspective, the dynamical equations for different Chern-Simons frameworks
	 pose significant analytical challenges, even in scenarios involving radial symmetry and static conditions \cite{kumar1986charged}.   In  the Abelian Chern-Simons model, the self-dual structure was discovered  in \cite{MR1050529, MR1050530} and  has catalyzed extensive further investigation. Many problems of existence can be transformed into studies of elliptic partial differential equations or systems of equations, with a particular focus on exploring topological and non-topological solutions \cite{MR1429096,MR1690951,MR1946331,MR2168004,MR3390935,MR3412153}.

	Partial differential equations on discrete graphs have recently garnered significant interest and are now widely applied in diverse fields such as image processing, social network analysis, bioinformatics, and machine learning. Recent advancements have extended some traditional methods for solving partial differential equations in Euclidean spaces to the context of graphs.
	
	Pioneering research by Grigor'yan et al.~\cite{MR3542963,grigor2016kazdan,MR3665801} introduced a variational method to examine  the Kazdan-Warner equation, the Yamabe-type equation,  and other nonlinear equations. The main goal of these studies was to demonstrate the existence of solutions. For a   range of partial differential equations,  further investigations have been thoroughly investigated on graphs.
	 Notably, the existence results for Yamabe type equations have been detailed in references such as~\cite{MR3759076,MR3767372,MR3910409}. Studies addressing  Kazdan-Warner equations are found in~\cite{zhang2018p,MR3648273,ge2020p,MR4416135,MR4753528}, and the results concerning Schrödinger equations appear in~\cite{MR3833747,MR4092834,MR4561113,MR4578204}. Additionally, significant findings related to the heat equations are documented in~\cite{MR3688855,MR4607589}.

	Define \(\Delta\) as the Laplacian operator. In this research, we analyze a generalized Chern-Simons equation described by
	\begin{equation}\label{eq11}
		\Delta u = \lambda e^u (e^u - 1)^{2p+1} + 4\pi \sum_{j=1}^M n_j \delta_{p_j},
	\end{equation}
	where \(\lambda\) is a positive constant, \(p\) is a non-negative integer, \(n_i\) for \(1 \leq i \leq M\) are positive integers, and \(p_i\) for \(1 \leq i \leq M\) are specific points with \(\delta_{p_j}\) denoting the Dirac delta function at \(p_j\).
	  Within Euclidean spaces, the classification of a solution \(u(x)\) to equation \eqref{eq11} depends on its behavior at infinity: if \(u(x)\) approaches 0 as \(|x|\) increases indefinitely, it is termed topological, and if \(u(x)\) declines to \(-\infty\) as \(|x|\) increases, it termed non-topological.

For \( p = 0 \) and \( n_i = 1 \) for all \( i = 1, 2, \dots, M \), Eq.(\ref{eq11}) simplifies to
\begin{equation}\label{eqc}
	\Delta u = \lambda e^u (e^u - 1) + 4\pi \sum_{j=1}^M \delta_{p_j}.
\end{equation}
This equation was examined by Caffarelli et al. \cite{MR1324400} and Tarantello \cite{MR1400816} within a doubly periodic setting or on the 2-torus in \(\mathbb{R}^2\), confirming solution presence. Similarly, Huang et al. \cite{huang2020existence} and  Hou et al. \cite{hou2022existence} explored Eq.(\ref{eqc}) on finite graphs, establishing solution existence. For \( p = 2 \) and \( n_i = 1 \) for each \( i = 1, 2, \dots, M \), Eq.(\ref{eq11}) modifies to
\begin{equation}\label{eqs}
	\Delta u = \lambda e^u (e^u - 1)^5 + 4\pi \sum_{j=1}^M \delta_{p_j}.
\end{equation}
Han \cite{MR3033571} determined the presence of multi-vortices for Eq.(\ref{eqs}) in a similarly doubly periodic region of \(\mathbb{R}^2\). Chao et al. \cite{CHAO2023126787} and  Hu \cite{MR4725973} documented multiple solution findings for Eq.(\ref{eqs}) on finite graphs. Additional research on Chern-Simons models on graphs has been reported in \cite{lu2021existence,  MR4781956, houkong2022existence, gao2022existence, chao2022existence, li2023topological, hua2023existence, huang2021mean}.

We describe the lattice $\mathbb{Z}^n$ for $n \geq 2$ as a discrete graph denoted by $\mathbb{Z}^n = (V, E)$. Here, $V$ denotes the set of vertices and $E$ the set of edges:
\[
V = \left\{x: x = (x_1, \ldots, x_n)  \text{ where }  x_i \in \mathbb{Z} \text{ for each } 1\leq i\leq n\right\}\subseteq \mathbb{R}^n,
\]
\[
E = \left\{ xy : x, y \in V \text{ such that } d(x,y) = 1 \right\},
\] where \[
d(x, y) = \sum_{i=1}^n \left| x_i - y_i \right|.
\]

Consider a finite subset $\Omega \subset V$. The boundary, denoted $\delta \Omega$, is defined as
\[
\delta \Omega := \{ y \notin \Omega : \exists x \in \Omega , y \sim x \},
\]
and we denote by $\bar{\Omega} = \Omega \cup \delta \Omega$ the closure of $\Omega$.

	Now, we introduce  function spaces and operators on graphs to prepare for the subsequent analysis. 
	  Let \(C(V)\) denote the collection of all functions on \(V\).
	   Similarly, for a finite subset \(\Omega\) of \(V\), the set \(C(\Omega)\) is defined as the collection of functions that are defined over \(\Omega\). For any \(u\in C(V)\),  the integration of $u$ on V is defined as
	 $$\int_{V}ud\mu=\sum_{x \in V}u(x).$$  Analogously, we define 
the  \(l^q\)-norm  \((1\leq q<\infty)\) of 
	 \(u\in C(V)\) by 
	\[\|u\|_{l^q\left(V\right)}=\left(\sum_{x \in V}|u(x)|^q\right)^{\frac{1}{q}},\]
	and the \(l^{\infty}\)-norm of  \(u\in C(V)\) by 
	\[\|u\|_{l^{\infty}(V)}=\sup _{x \in V}|u(x)|.\]
Denote \(x \sim y\) whenever  \(xy\) is an edge in \(E\). 	We also define the  following   seminorm:
	$$
	|u|_{1,q}:=\left(\sum_{x \in V} \sum_{y \sim x}|u(y)-u(x)|^q\right)^{\frac{1}{q}} .
	$$
 The Laplacian operator 	for any $u\in C(V)$ is identified as
	\[
	\Delta u(x) = \sum_{y\sim x} (u(y) - u(x)).
	\]
	Furthermore, for any $u$ and $v$ on the graph, we characterize the gradient form as follows:
	$$
	\Gamma(u, v)(x):=\frac{1}{2} \sum_{y \sim x}(u(y)-u(x))(v(y)-v(x)) .
	$$
If $u = v$, this simplifies to $\Gamma(u) = \Gamma(u, u)$.

For functions $u, v$ in $C(\bar{\Omega})$, a bilinear form $E_{\Omega}(u, v)$ is specified as
\begin{equation}\label{bfe}
E_{\Omega}(u, v) := \int_{\bar{\Omega}} \Gamma(u, v) d\mu = \frac{1}{2} \sum_{\substack{x, y \in \bar{\Omega} \\ x \sim y}} (u(y) - u(x))(v(y) - v(x)),
\end{equation}
and $E_{\Omega}(u) = E_{\Omega}(u, u)$, referred to as the Dirichlet energy of $u$. Lemma 2.2 in \cite{MR3833747} indicates that for $u \in C(V)$ and $v \in C(\bar{\Omega})$ with $v$ vanishing on $\partial \Omega$, the following equality  holds:
\begin{equation}\label{gri}
\int_{\Omega \cup \partial \Omega} \Gamma(u, v) d \mu = -\int_{\Omega} \Delta u \, v \, d \mu.
\end{equation}

Let \(d(x) = d(x, 0)\) represent distance from the origin. 	In this study, we focus primarily on global solutions to the following   equation defined on \(V\):
	\begin{equation}\label{eq12}
		\left\{\begin{aligned}
			&\Delta u = \lambda e^u (e^u - 1)^{2p+1} + 4\pi \sum_{j=1}^M n_j \delta_{p_j}, \\
			&\lim_{d(x) \to +\infty} u(x) = 0.
		\end{aligned}\right.
	\end{equation}
	Our goal is to establish a topological solution which  is also  maximal.  The primary finding is presented as follows:

\begin{theorem}\label{th1}

	For \( n \geq 2 \), Eq.\eqref{eq12} provides a topological solution \( u \) in \( l^{2p+2}(V) \) on \( V \), which is also a maximal solution.
\end{theorem}

In our proof, we employ methods from \cite{MR1320569,hua2023existence}. First, we  establish an iterative scheme, which yields a monotone sequence $\{u_k\}$ addressing the Dirichlet problem. Subsequently, we introduce the functional $J_{\Omega}(u)$ and demonstrate that $J_{\Omega}(u_k)$ is uniformly bounded. By using the discrete Gagliardo-Nirenberg-Sobolev inequality, we ensure that the sequence $\{\|u_k\|_{l^{2p+2}(\Omega)}\}$ is uniformly bounded. As we pass to the limit, we ascertain a solution on  $\Omega$. Utilizing the exhaustion method, we extend this solution to $V$. Additionally, we  establish the maximality of the solution.

	\section{Proof of Theorem \ref{th1}}

Initially, we recall the maximum principle. It a foundational result on graphs.

\begin{lemma}\label{fil} \label{le1}\cite{hua2023existence}
	Assume $\Omega$ is a finite subset of $V$. Let $g$ belong to $C(\bar{\Omega})$ with $g>0$. If  $f \in C(\bar{\Omega})$ satisfies these conditions:
	\[
	\begin{cases}
		(\Delta - g) f \geqslant 0 & \text{on } \Omega, \\
	f\leqslant 0 & \text{on } \partial \Omega.
	\end{cases}
	\]
	Then, $f$ must be non-positive throughout $\bar{\Omega}$.
\end{lemma}

We then apply the maximum principle to construct an iterative sequence.
Define $\Omega_0$ as a finite subset of $V$ encompassing $\{p_1, p_2, \ldots, p_M\}$. Additionally, define $\Omega$ as a connected, finite subset such that $\Omega_0 \subset \Omega \subset V$.
 Set $h = 4 \pi \sum_{j=1}^M n_j \delta_{p_j}$ and define $N = 4 \pi \sum_{j=1}^M n_j$. 
Select a constant $\Lambda > (2p+2)\lambda > 0$, initialize $u_0 = 0$, and proceed with the iterative scheme:
	
	\begin{equation}\label{ite}
	\left\{
	\begin{array}{l}
		\left(\Delta - \Lambda\right) u_k = \lambda e^{u_{k-1}}\left (e^{u_{k-1}} - 1\right)^{2p+1} + h - \Lambda u_{k-1} \text{ on } \Omega, \\
		u_k = 0 \text{ on } \partial \Omega.
	\end{array}
	\right.
	\end{equation}

	\begin{lemma}\label{le3}
		Consider the sequence $\{u_k\}$ established in (\ref{ite}). It follows that each $u_k$ in the sequence is uniquely determined and adheres to the order
		\[
		0 = u_0 \geqslant u_1 \geqslant u_2 \geqslant \cdots.
		\]
	\end{lemma}
	
	\begin{proof}
		First, we establish the following equation  for $u_1$:
		\begin{equation}\label{dir}
		\begin{cases}
			\left(\Delta - \Lambda\right) u_1 = h & \text{ on } \Omega, \\
			u_1 = 0 & \text{ on } \partial \Omega.
		\end{cases}
	\end{equation}
	By the variation method in Lemma 2.2 in \cite{houkong2022existence}, we find that (\ref{dir}) yields a unique solution.	Utilizing Lemma \ref{le1}, it follows that $u_1 \leqslant 0$. 
	
	Assuming that
		\[
		0 = u_0 \geqslant u_1 \geqslant u_2 \geqslant \cdots \geqslant u_i,
		\]
		and given that
		\[
		\lambda e^{u_i}\left(e^{u_i} - 1\right)^{2p+1} + h - \Lambda u_i \in l^2(\Omega),
		\]
		we can also guarantee the existence and uniqueness of $u_{i+1}$ by the variation method again. 
		
		Analyzing the iterative equation (\ref{ite}), we derive that
		\[
		\begin{aligned}
			\left(\Delta-\Lambda\right)\left(u_{i+1}-u_i\right) &= \lambda e^{u_{i}}\left (e^{u_{i}} - 1\right)^{2p+1}-\lambda e^{u_{i-1}}\left (e^{u_{i-1}} - 1\right)^{2p+1} -\Lambda(u_i-u_{i-1}) \\
			&\geqslant \{ \lambda e^{\xi}\left(e^{\xi}-1\right)^{2p}\left[(2p+2)e^{\xi}-1\right]- \Lambda\}(u_i - u_{i-1}) \\
			&\geqslant \Lambda(e^{2 \xi} - 1)(u_i - u_{i-1})\\ &\geqslant 0,
		\end{aligned}
		\]
	where \(\xi\) denotes a function constrained such that \(u_i \leqslant \xi\leqslant u_{i-1}\). It  ensures that \( u_i\geq u_{i+1} \), thus validating the lemma in accordance with Lemma \ref{le1}.

	\end{proof}

Next, we define the natural functional $J_{\Omega}(u)$ and demonstrate that $J_{\Omega}(u_k)$ decreases as $k$ increases. We introduce the following functional defined over $\Omega$: 

\begin{equation}\label{efd}
J_{\Omega}(u) = \frac{1}{2} E_{\Omega}(u) + \sum_{x \in \Omega} \left[ \frac{\lambda}{2p+2} (e^{u(x)} - 1)^{2p+2} + h(x) u(x) \right].
\end{equation}

\begin{lemma}\label{iul}
	Define the sequence $\{u_k\}$ as specified in Eq.(\ref{ite}). The sequence satisfies the decreasing inequality:
	\[
	C \geqslant J_{\Omega}(u_1) \geqslant J_{\Omega}(u_2) \geqslant\cdots \geqslant J_{\Omega}(u_k) \geqslant \cdots,
	\]
	where $C$ is a constant determined by the parameters $n$, $\lambda$, $p$, and $N$.
	
\end{lemma}

\begin{proof}

	By multiplying Eq.(\ref{ite}) by the difference \(u_k - u_{k-1}\) and performing integration  on the domain \(\Omega\), we derive:
	\begin{equation}\label{ito}
		\begin{aligned}
			&\sum_{x \in \Omega} \left[\left(\Delta u_k  - \Lambda u_k\right) \left(u_k- u_{k-1}\right)\right](x) \\
			= &
			\sum_{x \in \Omega} \left[ \lambda e^{u_{k-1}} \left(e^{u_{k-1}} - 1\right)^{2p+1} \left(u_k - u_{k-1}\right) 
		  - \Lambda u_{k-1} \left(u_k - u_{k-1}\right) + h \left(u_k - u_{k-1}\right) \right](x).
		\end{aligned}
	\end{equation}
Utilizing the  identity as specified in (\ref{gri}), we have 
	\begin{equation}
		\sum_{x \in \Omega} \Delta u_k \left(u_k - u_{k-1}\right) = -E_{\Omega}( u_k, u_k - u_{k-1}) =  E_{\Omega}(u_{k}, u_{k-1})-E_{\Omega}(u_k) .
	\end{equation}
Integrating this  with equation (\ref{ito}), we conclude 
	\begin{equation}\label{2.4}
		\begin{aligned}
		&E_{\Omega}(u_k) - E_{\Omega}( u_k, u_{k-1})  + \sum_{x \in \Omega} \Lambda \left(u_k - u_{k-1}\right)^2  \\
		&=	-\sum_{x \in \Omega} \left[ \lambda e^{u_{k-1}} \left(e^{u_{k-1}} - 1\right)^{2p+1} \left(u_k - u_{k-1}\right) 
		 + h(x) \left(u_k - u_{k-1}\right) \right].
		\end{aligned}
	\end{equation}

Now we introduce a concave function for $x\leq 0$: 
\[\eta(x)=\frac{\lambda}{2p+2}\left( e^x-1\right)^{2p+2}-\frac{\Lambda}{2}x^2.\]
We see that 
\[\eta\left(u_{k-1}\right)-\eta\left(u_k\right)\geqslant \eta^{\prime}\left(u_{k-1})(u_{k-1}-u_k\right)=\left[\lambda\left(e^{u_{k-1}}-1\right)^{2p+1}e^{u_{k-1}}-\Lambda u_{k-1}\right]\left(u_{k-1}-u_k\right),\]
which yields that 	
\begin{equation}\label{2.5}
\begin{aligned}	
\frac{\lambda}{2p+2}\left(e^{u_k}-1\right)^{2p+2} \leqslant& \frac{\lambda}{2p+2}\left(e^{u_{k-1}}-1\right)^{2p+2}+\frac{\Lambda}{2}\left(u_k-u_{k-1}\right)^2\\
&+\lambda e^{u_{k-1}}\left(e^{u_{k-1}}-1\right)^{2p+1}\left(u_k-u_{k-1}\right) .	
\end{aligned}
\end{equation}

It  follows from (\ref{bfe}) that  
\begin{equation}\label{2.6}
\begin{aligned}
	 \left|E_{\Omega}\left(u_{k}, u_{k-1}\right)\right| \leqslant & \frac{1}{2} \sum_{\substack{x, y \in \bar{\Omega} \\ x \sim y}} \left|(u_k(y) - u_{k}(x))(u_{k-1}(y) - u_{k-1}(x))\right|\\
	 \leqslant &\frac{1}{4} \sum_{\substack{x, y \in \bar{\Omega} \\ x \sim y}} \left[u_k(y) - u_k(x))\right]^2+\frac{1}{4} \sum_{\substack{x, y \in \bar{\Omega} \\ x \sim y}} \left[u_{k-1}(y) - u_{k-1}(x))\right]^2\\
	= & \frac{1}{2} E_{\Omega}\left(u_{k}\right)+\frac{1}{2} E_{\Omega}\left(u_{k-1}\right).
\end{aligned}
\end{equation}

Combining (\ref{2.4}), (\ref{2.5}) and  (\ref{2.6}),  we conclude that 
$$
J_{\Omega}\left(u_k\right) \leqslant J_{\Omega}\left(u_k\right)+\frac{\Lambda}{2}\left\|u_{k-1}-u_k\right\|_{l^2(\Omega)}^2 \leqslant J_{\Omega}\left(u_{k-1}\right) .
$$

Next, we estimate the upper bound for $J_{\Omega}(u_1)$. Observing that
$$
\begin{aligned}
	E_{\Omega}\left(u_1\right) & =\frac{1}{2} \sum_{\substack{x, y \in \bar{\Omega} \\ x \sim y}} \left[u_1(y) - u_1(x))\right]^2\\
	& \leqslant \sum_{\substack{x, y \in \bar{\Omega} \\ x \sim y}}\left(u_1(x)^2+u_1(y)^2\right)\\
	& \leqslant 4n\left\|u_1\right\|_{l^2(\Omega)}^2,
\end{aligned}
$$
and $\left|e^{u_1}-1\right|=1-e^{u_1} \leqslant-u_1$, we get 
$$
\begin{aligned}
	J_{\Omega}\left(u_1\right) & \leqslant \frac{1}{2} \cdot 4n\left\|u_1\right\|_{l^2(\Omega)}^2+\frac{\lambda}{2p+2} \sum_{x \in \Omega} u_1(x)^{2p+2}+\frac{1}{2} \sum_{x \in \Omega}\left[h(x)^2+u_1(x)^2\right] \\
	& =c_1+c_2\left(\left\|u_1\right\|_{l^2(\Omega)}^2+\left\|u_1\right\|^{2p+2}_{l^{2p+2}(\Omega)}\right),
\end{aligned}
$$
where constants $c_1$ and $c_2$ are determined only by the parameters $n$, $p$, $\lambda$, and $N$.

Upon multiplying Eq. (\ref{dir}) by $u_1$ and performing a summation over $\Omega$, we obtain
$$
E_{\Omega}(u_1) + \Lambda \sum_{x \in \Omega} u_1^2 = -\sum_{x \in \Omega} h u_1.
$$
From this, we deduce
$$
\Lambda \sum_{x \in \Omega} u_1(x)^2 \leqslant \frac{1}{2\Lambda} \sum_{x \in \Omega} h(x)^2 + \frac{\Lambda}{2} \sum_{x \in \Omega} u_1(x)^2.
$$
Therefore,
$$
\sum_{x \in \Omega} u_1(x)^2 \leqslant \frac{\|h\|_{\ell^2(V)}^2}{\Lambda^2},
$$
leading to
$$
\sum_{x \in \Omega} u_1(x)^{2p+2} \leqslant \left(\sum_{x \in \Omega} u_1(x)^2\right)^{p+1} \leqslant \left(\frac{\|h\|_{\ell^2(V)}^2}{\Lambda^2}\right)^{p+1}.
$$
We thus conclude that $J_{\Omega}(u_1) \leq C$, where $C$ is dependent only on $n$, $\lambda$, $p$, and $N$, thereby completing the proof.

\end{proof}

Next, by applying the discrete Gagliardo-Nirenberg-Sobolev inequality and invoking Lemma \ref{iul}, we determine the upper bound for $\|u_k\|_{l^{2p+2}(\Omega)}$. The foundation for this application is drawn from the proof presented in Theorem 4.1 in \cite{MR4095474}:

\begin{lemma} \cite{MR4095474}
	Assuming $n \geqslant 2, q > 1, \gamma \geqslant q$, and $q^{\prime} = \frac{q}{q-1}$, the inequality
	\[
	\|u\|_{l^{\frac{\gamma n}{n-1}}(V)}^{\gamma} \leq C(q, n, \gamma) |u|_{1,q} \|u\|^{\gamma-1}_{l^{(\gamma-1)q^{\prime}}(V)}.
	\]
	is satisfied for any function $u \in l^q(V)$.
\end{lemma}

Lemma 2.1 in \cite{MR3277179} implies that 

$$
\|u\|_{l ^{p^{\prime\prime}}\left(V\right)} \leqslant\|u\|_{l ^{p^{\prime}}\left(V\right)},
$$
for any $p^{\prime \prime }\geqslant p^{\prime}$.
Letting $\gamma=2(p+1)$, $q=2$ and $q^{\prime}=2$, we have for $u\in l^{2}(V)$, 
\begin{equation}\label{por}
	\|u\|_{l^{4p+4}\left(V\right)} \leqslant\|u\|_{l^{\frac{2 n(p+1)}{n-1}}\left(V\right)} \leqslant C(n,p)|u|_{1,2}^{\frac{1}{2p+2}}\|u\|_{l^{4p+2}\left(V\right)}^{\frac{2p+1}{2p+2}}.
\end{equation}

\begin{lemma}
	\label{ulp} 
	Define the sequence $\{u_k\}$ as specified in Eq.(\ref{ite}). For any index $k \geq 1$, the following inequality holds:
	\begin{equation}
		\|u_k\|_{l^{2p+2}(\Omega)} \leq C_2\left(J_{\Omega}(u_k) + 1\right) \leq C_1,
	\end{equation}
	where constants $C_2$ and $C_1$ are determined exclusively by the parameters $n$, $\lambda$, $p$, and $N$.
\end{lemma}

\begin{proof}
	
	Define $\tilde{u}_k$ as the null extension of $u_k$ across $V$:
	\begin{equation}\label{exd}
		\tilde{u}_k(x) = \left\{
		\begin{array}{l@{\quad\text{on}\ }l}
			u_k(x) & \Omega, \\
			0 & \Omega^c.
		\end{array}
		\right.
	\end{equation}
	It is evident that $\tilde{u}_k \in l^{2}(V)$. By (\ref{por}), the following inequality holds: 
		$$
	\|\tilde{u}_k\|_{l^{4p+4}(V)}^{2p+2} \leqslant (C(n,p))^{2p+2} |\tilde{u}_k|_{1,2} \|\tilde{u}_k\|_{l^{4p+2}(V)}^{2p+1} .
	$$
	Referring to (\ref{exd}), we obtain:
	$$
	\begin{aligned}
		\|\tilde{u}_k\|_{l^{4p+4}(V)}^{4p+4} &= \sum_{x \in \Omega} u_k(x)^{4p+4}, \\
		\|\tilde{u}_k\|_{l^{4p+2}(V)}^{4p+2} &= \sum_{x \in \Omega} u_k(x)^{4p+2},
	\end{aligned}
	$$
	and
	$$
	|\tilde{u}_k|_{1,2} \leqslant \left(2 E_{\Omega}(u_k)\right)^{\frac{1}{2}}.
	$$
	Thus, it follows that:
	\begin{equation}\label{foe}
	\sum_{x \in \Omega} u_k(x)^{4p+4} \leqslant C_3 E_{\Omega}(u_k) \sum_{x \in \Omega} u_k(x)^{4p+2},
	\end{equation}
	where $C_3 = 2(C(n,p))^{4p+4}$.
	
	Using the identity $1-e^{u_k} = 1-e^{-|u_k|} = \frac{e^{|u_k|}-1}{e^{|u_k|}} \geq \frac{|u_k|}{1+|u_k|}$, and (\ref{efd}), we deduce:
	
	$$
	\begin{aligned}
		J_{\Omega}(u_k)&=\frac{1}{2} E_{\Omega}(u_k)+\sum_{x \in \Omega}\left[\frac{\lambda}{2p+2}\left(e^{u_k(x)}-1\right)^{2p+2}+h(x) u_k(x)\right] \\
		& \geqslant \frac{1}{2} E_{\Omega}(u_k)+\frac{\lambda}{2p+2} \sum_{x \in \Omega}\left(\frac{|u_k(x)|}{1+|u_k(x)|}\right)^{2p+2}-\|h\|_{l^{\frac{4p+4}{4p+3}}\left(\Omega\right)}\|u_k\|_{l^{4p+4}(\Omega)} \\
		& \geqslant \frac{1}{2} E_{\Omega}(u_k)+\frac{\lambda}{2p+2} \sum_{x \in \Omega}\left(\frac{|u_k(x)|}{1+|u_k(x)|}\right)^{2p+2}-C_4\left(E_{\Omega}(u_k)\right)^{\frac{1}{4p+4}}\left(\sum_{x \in \Omega} u_k(x)^{4p+2}\right)^{\frac{1}{4p+4}},  
	\end{aligned}
	$$	
	where $C_4 $ is  a uniform constant solely based on $n$, $p$, and $N$

Let $\epsilon>0$ is a constant to be chosen later. By Yung's inequality, we have 
\[	\begin{aligned} &C_4\left(E_{\Omega}(u_k)\right)^{\frac{1}{4p+4}}\left(\sum_{x \in \Omega} u_k(x)^{4p+2}\right)^{\frac{1}{4p+4}}\\
	& =\left[ C_4\epsilon^{-\frac{2p+1}{2p+2}} \left(\frac{2p+2}{2p+1}\right)^{-\frac{2p+1}{2p+2}}\left(E_{\Omega}(u_k)\right)^{\frac{1}{4p+4}}\right]\left[\epsilon^{\frac{2p+1}{2p+2}} \left(\frac{2p+2}{2p+1}\right)^{\frac{2p+1}{2p+2}}\left(\sum_{x \in \Omega} u_k(x)^{4p+2}\right)^{\frac{1}{4p+4}}\right]\\
	&\leq \epsilon \|\tilde{u}_k\|_{l^{4p+2}(V)}+C_5 E_{\Omega} (u_k)^{\frac{1}{2}}\\
	&\leq \epsilon \|\tilde{u}_k\|_{l^{2p+2}(V)} +\frac{1}{4}E_{\Omega} (u_k)+C_6,
\end{aligned}\]
where $C_5$ and $C_6$ are solely influenced by $\epsilon$, $n$, $N$, and $p$.  

Hence, we obtain 
\begin{equation}\label{igi}	
	\begin{aligned}
	J_{\Omega}(u_k)
	\geqslant &\frac{1}{2} E_{\Omega}(u_k)+\frac{\lambda}{2p+2} \sum_{x \in \Omega}\left(\frac{|u_k(x)|}{1+|u_k(x)|}\right)^{2p+2}- \epsilon \|\tilde{u}_k\|_{l^{2p+2}(V)} -\frac{1}{4}E_{\Omega} (u_k)-C_6\\
	=&\frac{1}{4} E_{\Omega}(u_k)+\frac{\lambda}{2p+2} \sum_{x \in \Omega}\left(\frac{|u_k(x)|}{1+|u_k(x)|}\right)^{2p+2}- \epsilon \|u_k\|_{l^{2p+2}(\Omega)} -C_6.
\end{aligned}
\end{equation}

	Using the inequality given in (\ref{foe}), the subsequent estimate is derived:
	$$
		\begin{aligned}
			& \left(\sum_{x \in \Omega} u_k(x)^{2p+2}\right)^{2} = \left[\sum_{x \in \Omega} \left(\frac{|u_k(x)|}{1 + |u_k(x)|} \right)^{p+1} (1 + |u_k(x)|)^{p+1} |u_k(x)|^{p+1} \right]^{2} \\
			& \leqslant \sum_{x \in \Omega} \left(\frac{|u_k(x)|}{1 + |u_k(x)|}\right)^{2p+2} \sum_{x \in \Omega} (1 + |u_k(x)|)^{2p+2} u_k(x)^{2p+2} \\
			& \leqslant 2^{2p+2} \sum_{x \in \Omega} \left(\frac{|u_k(x)|}{1 + |u_k(x)|}\right)^{2p+2} \sum_{x \in \Omega} \left(u_k(x)^{2p+2} + u_k(x)^{4p+4}\right) \\
			& \leqslant 2^{2p+2} \sum_{x \in \Omega} \left(\frac{|u_k(x)|}{1 + |u_k(x)|}\right)^{2p+2} \sum_{x \in \Omega} u_k(x)^{2p+2} + 2^{2p+2}C_3\sum_{x \in \Omega} \left(\frac{|u_k(x)|}{1 + |u_k(x)|}\right)^{2p+2} E_{\Omega}(u_k)\sum_{x \in \Omega} u_k(x)^{4p+2} \\
			&\leq 2^{2p+2} \sum_{x \in \Omega} \left(\frac{|u_k(x)|}{1 + |u_k(x)|}\right)^{2p+2} \sum_{x \in \Omega} u_k(x)^{2p+2} + 2^{2p+2}C_3\sum_{x \in \Omega} \left(\frac{|u_k(x)|}{1 + |u_k(x)|}\right)^{2p+2} E_{\Omega}(u_k)\left(\sum_{x \in \Omega} u_k(x)^{2p+2}\right)^{\frac{2p+1}{p+1}}\\
			 &\leq  \frac{1}{4} \left(\sum_{x \in \Omega} u_k(x)^{2p+2}\right)^{2}
	          + C_7\left[\left(\sum_{x \in \Omega} \left(\frac{|u_k(x)|}{1 + |u_k(x)|}\right)^{2p+2}\right)^2+\sum_{x \in \Omega} \left(\frac{|u_k(x)|}{1 + |u_k(x)|}\right)^{2p+2} E_{\Omega}(u_k)\left(\sum_{x \in \Omega} u_k(x)^{2p+2}\right)^{\frac{2p+1}{p+1}}\right]\\
			 &\leq  \frac{1}{2} \left(\sum_{x \in \Omega} u_k(x)^{2p+2}\right)^{2} + C_8\left[\left(\sum_{x \in \Omega} \left(\frac{|u_k(x)|}{1 + |u_k(x)|}\right)^{2p+2}\right)^2+\left(\sum_{x \in \Omega} \left(\frac{|u_k(x)|}{1 + |u_k(x)|}\right)^{2p+2}\right)^{2p+2} E_{\Omega}(u_k)^{2p+2}\right]\\
			& \leqslant \frac{1}{2} \left(\sum_{x \in \Omega} u_k(x)^{2p+2}\right)^2 + C_9 \left[1 + \left(\sum_{x \in \Omega} \left(\frac{|u_k(x)|}{1 + |u_k(x)|}\right)^{2p+2}\right)^{4p+4} + E_{\Omega}(u_k)^{4p+4}\right],
		\end{aligned}
$$
which  results in 
\begin{equation}\label{ukp}
\|u_k\|_{l^{2p+2}(\Omega)} \leqslant C_{10}\left[1+\sum_{x \in \Omega}\left(\frac{|u_k(x)|}{1+|u_k(x)|}\right)^{2p+2}+E_{\Omega}(u_k)\right],
\end{equation}
where $C_7$-$C_{10}$ are constants depending only on $n$ and $p$.

Selecting $\epsilon=\frac{\min \left\{\frac{1}{8}, \frac{\lambda}{4p+4}\right\}}{C_{10}}$ and integrating the results from (\ref{igi}) and (\ref{ukp}), we obtain

$$
\|u_k\|_{l^{2p+2}(\Omega)} \leqslant C_2(J_{\Omega}(u_k)+1) .
$$
Furthermore, by Lemma \ref{iul}, we conclude 
$$
\|u_k\|_{l^{2p+2}(\Omega)} \leqslant C_2(J_{\Omega}(u_k)+1)\leq C_1,
$$	 
where 	$C_2$ and $C_1$ depend only on $n$, $\lambda$, $p$ and  $N$. 
\end{proof}

The boundedness of $\|u_k\|_{l^{2p+2}(\Omega)}$ ensures that Eq.(\ref{cho}) has a solution.

\begin{lemma}\label{mle}
	Define \(\Omega\) as a finite subset of \(V\) encompassing  points $\{p_1, p_2, \ldots, p_M\}$. There exists a function \(u_{\Omega}\) to the following problem:
	\begin{equation}\label{cho}
	\begin{cases}
		\Delta u = \lambda e^u (e^u - 1)^{2p+1} + h, & \text{in } \Omega, \\
		u(x) = 0, & \text{on } \delta \Omega.
	\end{cases}
\end{equation}
	This solution is a maximal  across all alternatives. Additionally, we have  \(\|u_{\Omega}\|_{l^{2p+2}(\Omega)} \leq C_1\), with \(C_1\) determined exclusively by \(n\), \(\lambda\), \(p\), and \(N\).
\end{lemma}

\begin{proof}

By applying Lemmas \ref{ulp} and \ref{le3} and noting that $l^{2p+2}(\Omega)$ is finite-dimensional, we get
\[
u_k \rightarrow u_{\Omega} \text{ in } l^{2p+2}(\Omega),
 \]
and
\[
\|u_{\Omega}\|_{l^{2p+2}(\Omega)} \leqslant C_0.
\]
Since the convergence is pointwise, we see that the function \(u_{\Omega}\) adheres to the equation:
\[
\begin{cases}
	\Delta u = \lambda e^u (e^u - 1)^{2p+1} + h, & \text{on } \Omega, \\
	u(x) = 0, & \text{on } \delta \Omega.
\end{cases}
\]

 The remaining task is to demonstrate that this solution is maximal. For a function \(U \in C(\bar{\Omega})\) satisfying
\[
\left\{
\begin{array}{l}
	\Delta U \geqslant \lambda e^U\left(e^U-1\right)^{2p+1} + h \text{ on } \Omega, \\
	U(x) \leqslant 0 \text{ on } \delta\Omega,
\end{array}
\right.
\]
we claim  that
\begin{equation}\label{cmp}
 u_0 \geqslant u_1 \geqslant \cdots \geqslant u_k \geqslant \cdots \geqslant u_{\Omega} \geqslant U.
\end{equation}
Initially, it is noted that
\[
\Delta U \geqslant \lambda e^U\left(e^U-1\right)^{2p+1} + h \geqslant \lambda e^U\left(e^U-1\right)^{2p+1}.
\]
We assert that \(\sup_{x \in \Omega} U(x) \leqslant 0\). Should this not hold, and \(U(x_0) = \sup_{x \in \Omega} U(x) > 0\) for some \(x_0 \in \Omega\), then
\[
0 \geqslant \Delta U(x_0) \geqslant \lambda e^{U(x_0)}\left(e^{U(x_0)}-1\right)^{2p+1}> 0,
\]
resulting in a contradiction and thus establishing the claim.

Assuming \(U \leqslant u_k\), then
\[
\begin{aligned}
	(\Delta-\Lambda)(u_{k+1}-U) & \leqslant \lambda e^{u_k}\left(e^{u_k}-1\right)^{2p+1}-\lambda e^U\left(e^U-1\right)^{2p+1}-\Lambda(u_k-U) \\
	& \leqslant \lambda e^{\xi}\left(e^{\xi}-1\right)^{2p}\left[(2p+2)e^{\xi}-1\right](u_k - U) - \Lambda(u_k - U) \\
	& \leqslant \Lambda\left(e^{2 \xi}-1\right)(u_k-U) \leqslant 0,
\end{aligned}
\]
where \(\xi\) is a function  which lies between  $U$ and $u_k$. This ensures that \(U \leqslant u_{k+1}\) by Lemma \ref{fil}. Hence $U\leq u_{\Omega}$. If $u^{\prime}_{\Omega}$ is a another solution, we conclude that $u^{\prime}_{\Omega}\leq u_{\Omega}$, which means that $u_{\Omega}$ is maximal. 
\end{proof}

Mow, we will prove Theorem 1.1.
Choose a series of finite and connected subsets \(\{\Omega_j\}\) such that $$\quad \bigcup_ {j=1}^{\infty} \Omega_j = V$$ and $$\Omega_0 \subset \Omega_1 \subset \cdots \subset \Omega_i \subset \cdots.$$

These lemmas will be utilized to prove Theorem 1.1. For any integers \(1 \leqslant i \leqslant l\), given that \(\Omega_i \subset \Omega_l\) and observing that \(u_{\Omega_l} \leqslant 0\) on \(\bar{\Omega}_i\), it follows from (\ref{cmp}) that
\[
u_{\Omega_l} \leqslant u_{\Omega_i} \quad \text{on } \bar{\Omega}_i.
\]

Let \(\tilde{u}_{\Omega_i}\) denote the null extension of \(u_{\Omega_i}\) to \(V\). Then, we have
\[
0 \geqslant \tilde{u}_{\Omega_1} \geqslant \tilde{u}_{\Omega_2} \geqslant \cdots \geqslant \tilde{u}_{\Omega_i} \geqslant \cdots
\]
across \(V\). Given that \(\left\|\tilde{u}_{\Omega_i}\right\|_{l^{2p+2}\left(V\right)} \leqslant C_1\) for all \(i \geqslant 1\), we obtain
\[
\tilde{u}_{\Omega_k}(x) \rightarrow \tilde{u}(x), \quad \forall x \in V,
\]
and \(\tilde{u}\in l^{2p+2}(V)\). Consequently, \(\tilde{u}\) satisfies the following equation:
\[
\begin{cases}
	\Delta \tilde{u} = \lambda e^{\tilde{u}} \left(e^{\tilde{u}} - 1\right)^{2p+1} + 4\pi \sum\limits_{j=1}^M n_j\delta_{p_j}, & \text{on } V, \\
	\lim_{d(x) \rightarrow +\infty} \tilde{u}(x) = 0.
\end{cases}
\]
Obviously, $\tilde{u}$ is   topological.

Suppose there exists another topological solution \(\bar{u}\) to (\ref{eq12}). For all \(x \in V\), we will show that \(\bar{u}(x) \leqslant 0\).  Assume, for contradiction, that there exists a point \(x_0 \in V\) such that   \(\bar{u}(x_0) > 0\). Given that \(\lim\limits_{d(x) \rightarrow +\infty} \bar{u}(x) = 0\), there must be a domain \(\Omega_i\) where \(\bar{u}(x_1) = \sup_{x \in \Omega_i} \bar{u}(x) > 0\) for some \(x_1 \in \Omega_i\). Therefore, we find
\[
0 \geq \Delta \bar{u}(x_1) \geqslant \lambda e^{\bar{u}(x_1)} \left(e^{\bar{u}(x_1)} - 1\right)^{2p+1} > 0,
\]
which leads to a contradiction.

Applying (\ref{cmp}) over \(\Omega_i\), we obtain
\[
\bar{u} \leqslant \tilde{u}_{\Omega_i}.
\]
Fix an integer \(l \geqslant 1\) and letting  \(i \geqslant l\), we conclude 
\[
\bar{u}(x) \leqslant \underline{\lim}_{i \rightarrow \infty} \tilde{u}_{\Omega_j}(x) = \tilde{u}(x) \quad 
\]
on $\Omega_l$. Consequently, we have \(\bar{u} \leqslant \tilde{u}\) across \(V\), indicating that  \(\tilde{u}\) is also maximal.

\vskip 30 pt
\noindent{\bf ACKNOWLEDGMENTS}

This work is   partially  supported by  the National Key Research and Development Program of China 2020YFA0713100 and by the National Natural Science Foundation of China (Grant No. 11721101).

\bibliographystyle{plain}
\bibliography{D:/Reference/CHS.bib}	
	
\end{document}